\documentclass[12pt,a4paper]{amsart}
\usepackage{amssymb,latexsym}

\usepackage{enumerate}

\usepackage{a4wide}

\theoremstyle{plain}
\newtheorem{theorem}{Theorem}[section]
\newtheorem{lemma}[theorem]{Lemma}
\newtheorem{proposition}[theorem]{Proposition}
\newtheorem{corollary}[theorem]{Corollary}
\newtheorem*{maintheorem}{Main Theorem}

\theoremstyle{remark}

\newcommand{\C}{\ensuremath{\mathbb{C}}}
\newcommand{\E}{\ensuremath{\mathbb{E}}}
\newcommand{\R}{\ensuremath{\mathbb{R}}}
\newcommand{\g}[1]{\ensuremath{\mathfrak{#1}}}
\DeclareMathOperator{\ad}{ad}
\DeclareMathOperator{\Ad}{Ad}

\DeclareMathOperator{\Exp}{Exp}

\begin{document}
\title[Polar actions on the complex hyperbolic plane]{Polar actions on the complex
hyperbolic plane}

%\date{Version of \today}

\author[J.~Berndt]{J\"{u}rgen Berndt}
\address{Department of Mathematics\\
         King's College London\\
         United Kingdom}
\email[J\"{u}rgen Berndt]{jurgen.berndt@kcl.ac.uk}

\author[J.\,C.~D\'{\i}az-Ramos]{Jos\'{e}~Carlos D\'{\i}az-Ramos}
\address{Department of Geometry and Topology\\
         University of Santiago de Compostela\\
         Spain}
\email[Jos\'{e}~Carlos D\'{\i}az-Ramos]{josecarlos.diaz@usc.es}
\thanks{The second author has been supported by a Marie-Curie
        European Reintegration Grant (PERG04-GA-2008-239162)
        and projects MTM2009-07756
        and INCITE09207151PR (Spain)}



\begin{abstract}
We classify the polar actions on the complex hyperbolic
plane $\C H^2$ up to orbit equivalence. Apart from the trivial and transitive polar actions, there are five polar actions of cohomogeneity one and four polar actions of cohomogeneity two.
\end{abstract}

\maketitle

%%%%%%%%%%%%%%%%%%%%%%%%%% Section %%%%%%%%%%%%%%%%%%%%%%%%%%%%%%%
\section{Introduction}

Let $M$ be a Riemannian manifold and denote by $I(M)$ its isometry
group. A connected closed subgroup $G$ of $I(M)$ is said to act
polarly on $M$ if there exists a connected closed submanifold
$\Sigma$ of $M$ that intersects all the orbits of $G$
orthogonally. Thus, for each $p \in M$ the intersection of
$\Sigma$ and the orbit $G \cdot p$ of $G$ containing $p$ is
nonempty, and for all $p \in \Sigma$ the tangent space
$T_p\Sigma$ of $\Sigma$ at $p$ is contained in the normal
space $\nu_p(G\cdot p)$ of $G \cdot p$ at $p$. The submanifold
$\Sigma$ is called a section of the action.

Polar actions on Riemannian symmetric spaces of compact type are
understood reasonably well, see \cite{Ko07}, \cite{Ko09} and
\cite{PT99} for more details. On the other hand, due to the possible
noncompactness of the groups, polar actions on Riemannian symmetric
spaces of noncompact type are not understood except for the real
hyperbolic spaces. The purpose of this paper is to classify the polar
actions on the complex hyperbolic plane $\C H^2$ up to orbit
equivalence. This is the first complete such classification on a
nontrivial Riemannian symmetric space of noncompact type. We hope
that this investigation will provide further insight into the
structure theory of polar actions.

The complex hyperbolic plane is a Riemannian symmetric space of noncompact type, namely $\C H^2 = G/K$ with $G = SU(1,2)$ and $K = S(U(1)U(2))$. Denote by $o \in \C H^2$ the unique fixed point of the $K$-action on $\C H^2$ and by $\g{g} = \g{k} \oplus \g{p}$ the corresponding Cartan decomposition of the Lie algebra $\g{g}$ of $G$. Denote by $\theta \in {\rm Aut}(\g{g})$ the corresponding Cartan involution. Let $\g{a}$ be a maximal abelian subspace of $\g{p}$ and
$\g{g} = \g{g}_{-2\alpha}\oplus\g{g}_{-\alpha}\oplus\g{g}_0\oplus
\g{g}_\alpha\oplus\g{g}_{2\alpha}$
the corresponding restricted root space decomposition of $\g{g}$. The root space $\g{g}_0$ decomposes into $\g{g}_0 = \g{k}_0 \oplus \g{a}$, where $\g{k}_0$ is the centralizer of $\g{k}$ in $\g{a}$. The complex structure on $\C H^2$ leaves the root space $\g{g}_\alpha$ invariant, and therefore $\g{g}_\alpha \cong \C$. By $\g{g}_\alpha^{\R}$ we denote a real form of $\g{g}_\alpha$, that is, a real one-dimensional linear subspace of $\g{g}_\alpha$.

The subalgebra $\g{n} = \g{g}_\alpha \oplus \g{g}_{2\alpha}$ is nilpotent and the action of the connected closed subgroup $N$ of $G$ with Lie algebra $\g{n}$ on $\C H^2$ induces a foliation of $\C H^2$ by horospheres. On a horosphere there are two distinguished types of horocycles, those which are generated by a real form $\g{g}_\alpha^{\R}$ and those which are generated by $\g{g}_{2\alpha}$. In the first case the horocycle lies in a totally geodesic real hyperbolic plane $\R H^2 \subset \C H^2$, and in the second case the horocycle lies in a totally geodesic complex hyperbolic line $\C H^1 \subset \C H^2$. We call such horocycles real and complex, respectively.  The subalgebra $\g{n}$ is isomorphic to the Heisenberg algebra, and every horosphere in $\C H^2$ with the induced metric is isometric to the $3$-dimensional Heisenberg group with a suitable left-invariant Riemannian metric. The subalgebra  $\g{g}_\alpha^{\R} \oplus\g{g}_{2\alpha}$ of $\g{n}$ is abelian and the orbit through $o$ of the corresponding connected closed subgroup of $N$ is a Euclidean plane $\E^2$ embedded in the horosphere as a non-totally geodesic minimal surface.

\begin{maintheorem}
For each of the subalgebras $\g{h}$ of $\g{s}\g{u}(1,2)$ listed below the connected closed subgroup $H$ of $SU(1,2)$ with Lie algebra $\g{h}$ acts polarly on $\C H^2$:
\begin{enumerate}[{\rm(i)}]
\item Actions of cohomogeneity one - the section $\Sigma$ is a totally geodesic real hyperbolic line $\R H^1 \subset \C H^2$:\label{It:cohom1}
\begin{enumerate}[{\rm(a)}]
\item $\g{h} = \g{k} = \g{s}(\g{u}(1)\oplus\g{u}(2))\cong\g{u}(2)$; the orbits are $\{o\}$ and the distance spheres centered at $o$;

\item $\g{h} = \g{g}_{-2\alpha} \oplus\g{g}_0 \oplus\g{g}_{2\alpha} = \g{s}(\g{u}(1,1) \oplus \g{u}(1)) \cong \g{u}(1,1)$; the orbits are a totally geodesic complex hyperbolic line $\C H^1 \subset \C H^2$ and the tubes around $\C H^1$;

\item $\g{h} = \theta(\g{g}_\alpha^{\R}) \oplus \g{a} \oplus \g{g}_\alpha^{\R} \cong \g{so}(1,2)$; the orbits are a totally geodesic real hyperbolic plane $\R H^2 \subset \C H^2$ and the tubes around $\R H^2$;

\item $\g{h} = \g{k}_0 \oplus \g{g}_\alpha\oplus\g{g}_{2\alpha}$ or $\g{h} = \g{g}_\alpha\oplus\g{g}_{2\alpha}$; the orbits form a foliation of $\C H^2$ by horospheres;

\item $\g{h}=\g{a}\oplus \g{g}_\alpha^{\R}
    \oplus\g{g}_{2\alpha}$; the orbits form a foliation of
    $\C H^2$; one of its leaves is the minimal ruled real
    hypersurface of $\C H^2$ generated by a real horocycle in
    $\C H^2$, and the other leaves are the equidistant
    hypersurfaces.
\end{enumerate}

\item Actions of cohomogeneity two - the section $\Sigma$ is a totally geodesic real hyperbolic plane $\R H^2 \subset \C H^2$:\label{It:cohom2}
\begin{enumerate}[{\rm(a)}]
\item $\g{h} = \g{k} \cap (\g{g}_{-2\alpha} \oplus\g{g}_0 \oplus\g{g}_{2\alpha}) = \g{s}(\g{u}(1)\oplus\g{u}(1)\oplus\g{u}(1)) \cong \g{u}(1) \oplus \g{u}(1)$; the orbits are obtained by intersecting the orbits of the two cohomogeneity one actions \textup{(a)} and \textup{(b)} in \textup{(i)}: the action has one fixed point $o$, and on each distance sphere centered at $o$ the orbits are two circles  as singular orbits and $2$-dimensional tori as principal orbits;\label{It:cohom2:point}

\item $\g{h} = \g{g}_0$; the action leaves a totally geodesic $\C H^1 \subset \C H^2$ invariant. On this $\C H^1$ the action induces a foliation by a totally geodesic real hyperbolic line $\R H^1 \subset \C H^1$ and its equidistant curves in $\C H^1$. The other orbits are $2$-dimensional cylinders whose axis is one of the curves in that $\C H^1$; \label{It:cohom2:a}

\item $\g{h} = \g{k}_0\oplus\g{g}_{2\alpha}$; the orbits are obtained by intersecting the orbits of the two cohomogeneity one actions \textup{(b)} and \textup{(d)} in \textup{(i)}: the action leaves a horosphere foliation invariant, and on each horosphere the orbits consist of a complex horocycle and the tubes around it; \label{It:cohom2:g2a}

\item $\g{h} = \g{g}_\alpha^{\R} \oplus\g{g}_{2\alpha}$; the orbits are obtained by intersecting the orbits of the two cohomogeneity one actions \textup{(d)} and \textup{(e)} in \textup{(i)}: the action leaves a horosphere foliation invariant, and on each horosphere the action induces a foliation for which the minimally embedded Euclidean plane $\E^2$ and its equidistant surfaces are the leaves. \label{It:cohom2:foliation}
\end{enumerate}
\end{enumerate}
Every polar action on $\C H^2$ is either trivial, transitive, or orbit equivalent to one of the polar actions described above.
\end{maintheorem}

The paper is organized as follows. In Section 2 we summarize some
basic material, and in Section 3 we present the proof of the Main
Theorem. The only two interesting cases arise for cohomogeneity one
and cohomogeneity two. The cohomogeneity one case was settled in
\cite{BT06}, and the cohomogeneity two case for actions without
singular orbits in \cite{BD}. The main contribution of this paper to
the classification is the analysis of the cohomogeneity two case with
singular orbits.

%%%%%%%%%%%%%%%%%%%%%%%%%%%%%%%%%%%%%%%%%%%%%%%%%%%%%%%%%%%%%%%%%%%%%%%%%%%%
\section{Preliminaries}\label{S:preliminaries}

We refer to~\cite{BTV95} for more information. We denote by $\C
H^2=SU(1,2)/S(U(1)U(2))$ the complex hyperbolic plane with constant
holomorphic sectional curvature~$-1$. Define $G=SU(1,2)$ and denote by
$K\cong S(U(1)U(2))$ the isotropy group of $G$ at some point $o\in\C H^2$. The Cartan
decomposition of $\g{g}$ with respect to $o$ is
$\g{g}=\g{k}\oplus\g{p}$, where $\g{g}$ and $\g{k}$ are the Lie
algebras of $G$ and $K$ respectively, and $\g{p}$ is the orthogonal
complement of $\g{k}$ in $\g{g}$ with respect to the Killing form
$B$ of $\g{g}$. Let $\theta$ be the corresponding Cartan involution. Then
$\langle X,Y\rangle=-B(\theta X,Y)$ defines a positive definite inner
product on $\g{g}$ that satisfies $\langle \ad(X)Y,Z\rangle=-\langle
Y,\ad(\theta X)Z\rangle$ for all $X$, $Y$, $Z\in\g{g}$. As usual,
$\ad$ and $\Ad$ will denote the adjoint maps of $\g{g}$ and $G$,
respectively. It is customary to identify $\g{p}$ with the tangent
space $T_o\C H^2$.

A maximal abelian subspace $\g{a}$ of $\g{p}$ is $1$-dimensional and
induces a restricted root space decomposition
$\g{g}=\g{g}_{-2\alpha}\oplus\g{g}_{-\alpha}\oplus\g{g}_0\oplus
\g{g}_\alpha\oplus\g{g}_{2\alpha}$, where
$\g{g}_\lambda=\{X\in\g{g}:\ad(H)X=\lambda(H)X\text{ for all
}H\in\g{a}\}$. Recall that
$[\g{g}_\lambda,\g{g}_\mu]=\g{g}_{\lambda+\mu}$,
$\theta\g{g}_\lambda=\g{g}_{-\lambda}$, and
$\g{g}_0=\g{k}_0\oplus\g{a}$, where $\g{k}_0=\g{g}_0\cap\g{k}$. Note
that $\g{k}_0$ is isomorphic to $\g{u}(1)$ and
$\g{g}_{2\alpha}$ is $1$-dimensional. Let
$\g{n}=\g{g}_\alpha\oplus\g{g}_{2\alpha}$, which is a nilpotent
subalgebra of $\g{g}$ isomorphic to the $3$-dimensional Heisenberg algebra. Then
$\g{g}=\g{k}\oplus\g{a}\oplus\g{n}$ is an Iwasawa decomposition of
$\g{g}$ and the connected subgroup $AN$ of $G$ whose Lie algebra is
$\g{a}\oplus\g{n}$ acts simply transitively on $\C H^2$. We endow
$AN$, and hence $\g{a}\oplus\g{n}$, with the left-invariant metric
$\langle\,\cdot\,,\,\cdot\,\rangle_{AN}$, and the complex structure
$J$ that make $\C H^2$ and $AN$ isometric. This implies that
$\langle{X},{Y}\rangle_{AN} =\langle X_{\g{a}},Y_{\g{a}}\rangle
+\frac{1}{2}\langle X_{\g{n}},Y_{\g{n}}\rangle$ for $X,Y\in\g{a}\oplus\g{n}\cong T_{1}AN$,
where the subscript means orthogonal
projection. The complex structure $J$ on $\g{a}\oplus\g{n}$ satisfies
that $J\g{g}_\alpha = \g{g}_\alpha$  and $J\g{a}=\g{g}_{2\alpha}$.
Let $B$ be a unit vector in $\g{a}$ and define
$Z=JB\in\g{g}_{2\alpha}$. Note that $\langle B,B\rangle=\langle
B,B\rangle_{AN}=1$, whereas $\langle Z,Z\rangle=2\langle
Z,Z\rangle_{AN}=2$. Then
\[
[aB+U+xZ,bB+V+yZ]=
-\frac{b}{2}U+\frac{a}{2}V
+\left(-bx+ay+\frac{1}{2}\langle JU,V\rangle\right)Z,
\]
where $a$, $b$, $x$, $y\in\R$, and $U$, $V\in\g{g}_\alpha$.
Finally, we define
$\g{p}_\lambda=(1-\theta)\g{g}_\lambda\subset\g{p}$. Then,
$\g{p}=\g{a}\oplus\g{p}_\alpha\oplus\g{p}_{2\alpha}$, $\g{p}_\alpha$
is complex, and $\g{p}_{2\alpha}$ is one-dimensional. If $i$ denotes
the complex structure of $\g{p}$, we have
$iB=\frac{1}{2}(1-\theta)Z$, and $i(1-\theta)U=(1-\theta)JU$.

%%%%%%%%%%%%%%%%%%%%%%%%%%%%%%%%%%%%%%%%%%%%%%%%%%%%%%%%%%%%%%%%%%%%%%%%%%%%
\section{Proof of the Main Theorem}\label{S:proof}

For a Riemannian manifold $M$ we denote by $I(M)$ the isometry group of $M$ and by $T_pM$ the tangent space of $M$ at $p \in M$. If $\Sigma$ is a submanifold of $M$, we denote by $\nu_p\Sigma$ the normal space of $\Sigma$ at $p \in \Sigma$. For a subgroup $H \subset I(M)$ we denote by $H \cdot p$ the orbit of the $H$-action on $M$ containing $p$. We first recall a result from \cite{DK11}.

\begin{proposition}\label{P:criterion}
Let $M$ be a complete connected Riemannian manifold and $\Sigma$ be  a
connected totally geodesic embedded submanifold of $M$. A closed subgroup
$H$ of $I(M)$ acts polarly on $M$ with section $\Sigma$
if and only if there exists a point $o\in M$ such that
\begin{enumerate}[{\rm(a)}]
\item $T_o\Sigma\subset\nu_o(H\cdot
    o)$, \label{P:criterion:tangent}

\item the slice representation of $H_o$ on $\nu_o(H\cdot o)$ is
    polar and $T_o\Sigma$ is a section,\label{P:criterion:slice}

\item $\nabla_v X^*\in\nu_o\Sigma$ for all $v\in T_o\Sigma$ and
    all $X\in\g{h}$, where $X^*$ denotes the smooth vector field
    on $M$ defined by $X^*_p=\frac{d}{dt}_{\vert t=0}\Exp(tX)(p)$
    for each $p\in M$.\label{P:criterion:derivative}
\end{enumerate}
\end{proposition}

We will use a refinement of this result for symmetric spaces of noncompact type. Let $M = G/K$ be a Riemannian symmetric space of noncompact type, where $G = I^o(M)$ is the connected component of $I(M)$ containg the identity transformation of $M$ and $K$ is the isotropy subgroup of $G$ at $o \in M$. Let $\g{g}$ be the Lie algebra of $G$, $B$ the Killing form of $\g{g}$, and $\theta$ the Cartan involution of the Cartan decomposition $\g{g} = \g{k} \oplus \g{p}$. The inner product defined by $\langle X , Y \rangle = -B(X,\theta Y)$ for all $X,Y \in \g{g}$ is positive definite. We identify $T_oM$ with $\g{p}$ in the usual way.

\begin{corollary}\label{C:criterion}
Let $M=G/K$ be a Riemannian symmetric space of noncompact type, and let
$\Sigma$ be a connected totally geodesic submanifold of $M$
with $o \in \Sigma$. A connected closed subgroup $H$ of $I(M)$
acts polarly on $M$ with section $\Sigma$ if and only if
$T_o\Sigma\subset\nu_o(H\cdot o)$, $T_o\Sigma$ is a section of the
slice representation of $H_o$ on $\nu_o(H\cdot o)$, and
\[
\langle[v,w],X\rangle= -B([v,w],\theta X) = 0\text{ for all $v$, $w\in T_o\Sigma \subset \g{p}$ and all $X\in\g{h}$.}
\]
\end{corollary}

\begin{proof}
Every totally geodesic submanifold in $G/K$ is embedded. Conditions~(\ref{P:criterion:tangent})
and~(\ref{P:criterion:slice}) of Proposition~\ref{P:criterion} are
satisfied by hypothesis, so we only have to check
condition~(\ref{P:criterion:derivative}). Let $v\in T_o\Sigma$ and
$X\in\g{h}$. Then, $v$ can be considered as a vector in
$\g{p}\subset\g{g}$, and hence we have $\nabla_v
X^*=[v^*,X^*]_o=-[v,X]^*_o=-[v,X]_{\g{p}}$, where the subscript means
orthogonal projection onto $\g{p}$ (see for
example~\cite[\S~IV.6]{S96}). Therefore, $\nabla_v X^*\in
\nu_o\Sigma$ if and only if for each $w\in T_o\Sigma\subset\g{p}$ we
have $0=\langle \nabla_v X^*,w\rangle=-\langle
[v,X]_{\g{p}},w\rangle=\langle X,[v,w]\rangle$.
\end{proof}

Assume that $H$ is a connected closed subgroup of $SU(1,2)$ acting
polarly on $\C H^2$, and let $\Sigma$ be a section of the action of
$H$. Since $\Sigma$ is totally geodesic, it is congruent to a point,
a geodesic which we can view as a totally geodesic $\R H^1$, a
totally geodesic complex hyperbolic line $\C H^1$, a totally geodesic
real hyperbolic plane $\R H^2$, or the whole complex hyperbolic
plane $\C H^2$. Clearly, if $\Sigma$ is a point, then the action of $H$ is
transitive, and if $\Sigma$ is the entire space, then the action is
trivial. So the only possibilities left are $\R H^1$, $\R H^2$, and
$\C H^1$.

If $\Sigma = \R H^1$, then the action of $H$ is of
cohomogeneity one (and also hyperpolar). Cohomogeneity one actions on
complex hyperbolic spaces were classified in~\cite{BT06}. A more
geometric classification in terms of the constancy of the principal
curvatures of a real hypersurface in $\C H^2$ can be found
in~\cite{BD07}. This corresponds to item~(\ref{It:cohom1}) of the
Main Theorem.

Therefore, the only remaining possibility for $\Sigma$ is to be an $\R H^2$ or a $\C
H^1$, which both have dimension~$2$. Hence, from now on we assume that $H$
acts on $\C H^2$ with cohomogeneity~$2$.

If all the orbits of the action of $H$ have the same dimension, then
there are no exceptional orbits and $H$ induces a homogeneous polar
foliation of $\C H^2$~\cite{BDT11}. Homogeneous polar foliations of
complex hyperbolic spaces were classified by the authors
in~\cite{BD}. This corresponds to case~(\ref{It:cohom2:foliation}) of
the Main Theorem.

Thus we can assume from now on that the action of $H$ has a singular
orbit. For dimension reasons, this orbit can only be $0$-dimensional
or $1$-dimensional. Assume first that there is a $0$-dimensional
orbit, that is, there is a point $o\in\C H^2$ that is fixed by the
action of $H$. In this case the group $H$ has to be compact. Indeed,
let $\{h_n\}$ be a sequence contained in $H$. Since $H$ fixes $o$, we
have that $\{h_n(o)\}$ converges to $o$. Since the group is closed in
$SU(1,2)$, the action of $H$ is proper and hence, by definition of
proper action, $\{h_n\}$ has a convergent subsequence. This shows
that $H$ is compact. In any case, polar actions with a fixed point on
$\C H^2$ have been classified in~\cite[Proposition 12~(ii)]{DK11}.
There are exactly three possibilities up to orbit equivalence: the
trivial action, the isotropy action of $S(U(1)U(2))$ (which is of
cohomogeneity one), and the action of $S(U(1)U(1)U(1))\cong U(1)\cdot
U(1)$, which is of cohomogeneity two and corresponds to
case~(\ref{It:cohom2:point}) of the Main Theorem. It is worthwhile to
point out at this stage that polar actions with a fixed point in $\C
H^2$ correspond to polar actions in $\C P^1$. The only nontrivial and
nontransitive polar action on $\C P^1$ up to orbit equivalence is the
isotropy action of $U(1)\cong S(U(1)U(1))$, which has two fixed points
as singular orbits; the rest of the orbits are principal, and in
particular one of its orbits is a totally geodesic $\R P^1$ in $\C
P^1$. This action is orbit equivalent to the action of $SO(2)$ on $\C
P^1$.

Finally, let us assume that $H$ has a singular orbit of dimension $1$
and no fixed points. Let $\g{h}$ be the Lie algebra of $H$. Let
$\g{l}$ be a proper maximal subalgebra of $\g{su}(1,2)$ containing
$\g{h}$. It is known that $\g{l}$ is either reductive or parabolic
(see~\cite{M61} or~\cite[Theorem~3.2]{BT} for a more detailed proof).

Assume first that $\g{l}$ is reductive. Then, up to conjugation,
$\g{l}$ is $\g{s}(\g{u}(1,1) \oplus \g{u}(1))\cong\g{su}(1,1)$,
$\g{so}(1,2)$, or $\g{s}(\g{u}(1)\oplus\g{u}(2))\cong\g{u}(2)$. The
last possibility corresponds to a compact group and hence $H\subset
S(U(1)U(2))$ would have a fixed point by Cartan's fixed point
theorem, contradicting our assumption. Then
$\g{l}=\g{s}\g{u}(1,1)$ or $\g{l}=\g{so}(1,2)$. In
both cases $\g{l}$ has dimension~$3$, and the action of $L$, the
connected Lie subgroup of $SU(1,2)$ whose Lie algebra is $\g{l}$, is
of cohomogeneity one. Thus, $\dim\g{h}<3$. By the classification of
Lie algebras of low dimension, this implies that $\g{h}$ is solvable,
and hence it is contained in a Borel subalgebra $\g{b}$, that is, a
maximal solvable subalgebra of $\g{su}(1,2)$. There are, up to
conjugation, exactly two types of Borel subalgebras in $\g{su}(1,2)$:
of maximally compact type, and of maximally noncompact type. Again,
$\g{h}$ cannot be contained in a Borel subalgebra of maximally
compact type, because such a subalgebra $\g{b}$ is compact and hence
$H$ would have a fixed point by Cartan's fixed point theorem. Hence
$\g{h}$ is contained in a Borel subalgebra of maximally noncompact
type. Then, with respect to a suitable Cartan decomposition
$\g{su}(1,2)=\g{k}\oplus\g{p}$, and a suitable maximal abelian
subspace $\g{a}$ of $\g{p}$, we have
$\g{b}=\g{t}\oplus\g{a}\oplus\g{g}_\alpha\oplus\g{g}_{2\alpha}$. Here
$\g{t}\oplus\g{a}$ is a Cartan subalgebra of $\g{su}(1,2)$;
$\g{a}\subset\g{p}$ is called is vector part, and $\g{t}\subset\g{k}$
is called the toroidal part. It is easy to see in this case that
$\g{t}=\g{k}_0$. Hence,
$\g{b}=\g{k}_0\oplus\g{a}\oplus\g{g}_\alpha\oplus\g{g}_{2\alpha}$
turns out to be a parabolic subalgebra. Thus, we may assume from now
on that the maximal subalgebra $\g{l}$ containing $\g{h}$ is
parabolic. Write, as before, this parabolic subalgebra as
$\g{l}=\g{k}_0\oplus\g{a}\oplus\g{g}_\alpha\oplus\g{g}_{2\alpha}$.

Since a subgroup of $SU(1,2)$ whose Lie algebra is contained in
$\g{a}\oplus\g{g}_\alpha\oplus\g{g}_{2\alpha}$ induces a foliation on
$\C H^2$, we conclude that the orthogonal projection of $\g{h}$ onto
$\g{k}_0$ is nonzero. Moreover, we know that $\g{k}_0$ is
$1$-dimensional, and that the orbit of $H$ through the origin $o$ is
at most $2$-dimensional, which implies that the orthogonal projection
of $\g{h}$ onto $\g{a}\oplus\g{g}_\alpha\oplus\g{g}_{2\alpha}$ is at
most $2$-dimensional. Therefore, $\g{h}$ can be written as
$\g{h}=\g{k}_0\oplus\R\xi\oplus\R\eta$, where $\xi$,
$\eta\in\g{a}\oplus\g{g}_\alpha\oplus\g{g}_{2\alpha}$ are linearly
independent vectors, or $\g{h}=\R(T+\xi)\oplus\R\eta$, with
$T\in\g{k}_0$, $T\neq 0$, and $\xi$,
$\eta\in\g{a}\oplus\g{g}_\alpha\oplus\g{g}_{2\alpha}$. We analyze
both possibilities.

Assume first that $\g{h}=\g{k}_0\oplus\R\xi\oplus\R\eta$, where
$\xi$, $\eta\in\g{a}\oplus\g{g}_\alpha\oplus\g{g}_{2\alpha}$ are
linearly independent vectors. It follows from the properties of root
spaces, the fact that $\g{h}$ is a Lie algebra, and the skew-symmetry
of the elements of $\ad(\g{k}_0)$, that
$\ad(\g{k}_0)\xi\in\g{g}_\alpha\cap(\g{h}\ominus\R\xi)=\R\eta$,
$\ad(\g{k}_0)\eta\in\g{g}_\alpha\cap(\g{h}\ominus\R\eta)=\R\xi$.
Since $\langle\ad(T)\xi,\eta\rangle=-\langle\ad(T)\eta,\xi\rangle$
for each $T\in\g{k}_0$, $\ad(\g{k}_0)\xi$ and $\ad(\g{k}_0)\eta$ are
both zero, or both nonzero. If $\ad(\g{k}_0)\xi=\ad(\g{k}_0)\eta=0$,
we conclude that $\xi$, $\eta\in\g{a}\oplus\g{g}_{2\alpha}$, so
$\g{h}=\g{k}_0\oplus\g{a}\oplus\g{g}_{2\alpha}$. This is not possible
because the corresponding group $H$ would act with cohomogeneity one.
Let us assume then that $\ad(\g{k}_0)\xi$ and $\ad(\g{k}_0)\eta$ are
both nonzero. Hence, $\R\xi\oplus\R\eta\subset\g{g}_\alpha$, and
since they are linearly independent and $\g{g}_\alpha$ is
$2$-dimensional, it follows that $\g{h}=\g{k}_0\oplus\g{g}_\alpha$.
This is not possible because $\g{k}_0\oplus\g{g}_\alpha$ is not a Lie
algebra.

In order to deal with the second possibility we start first with

\begin{lemma}\label{L:dim2}
Assume $\g{h}=\R(T+\xi)\oplus\R\eta$, with $0 \neq T\in\g{k}_0$
and $\xi$, $\eta\in\g{a}\oplus\g{g}_\alpha\oplus\g{g}_{2\alpha}$.
Then $\g{h}$ can be written in one of the following forms:
\begin{enumerate}[{\rm(a)}]
\item $0 \neq \xi\in\g{g}_\alpha$ and $0 \neq \eta\in\g{g}_{2\alpha}$, or\label{L:dim2:a}

\item $\xi=0$ and $0 \neq \eta\in\g{a}\oplus\g{g}_{2\alpha}$, or\label{L:dim2:b}

\item $\xi=[T,Y] +Z$ and $\eta=2B+Y+dZ$, where $d \in \R$ and $0 \neq Y\in\g{g}_\alpha$ such that $[[T,Y],Y] = 2Z$.\label{L:dim2:c}
\end{enumerate}
\end{lemma}

\begin{proof}
Write $\xi=aB+X+bZ$, and $\eta=cB+Y+dZ$, with $a$, $b$, $c$,
$d\in\R$, and $X$, $Y\in\g{g}_\alpha$. We may assume that $\langle
X,Y\rangle=0$.

First of all, by the algebraic properties of root spaces,
$[T+\xi,\eta]\in(\g{g}_\alpha\oplus\g{g}_{2\alpha})\cap\g{h}\subset\R\eta$.
We can therefore write $\lambda\eta=[T+\xi,\eta]$ for some $\lambda \in \R$. Inserting
the above expressions for $\xi$ and $\eta$, and taking the components of the resulting expression  in $\g{a}$,
$\g{g}_\alpha$ and $\g{g}_{2\alpha}$ we get
\begin{align}
\lambda c &=0,\label{E:lc}\\[1ex]
\lambda Y &=\frac{a}{2}Y-\frac{c}{2}X+[T,Y],\label{E:XY}\\
\lambda d &= ad-bc+\frac{1}{2}\langle [X,Y],Z \rangle.\label{E:abcd}
\end{align}

We consider the two cases $Y = 0$ and $Y \neq 0$ separately.

Case 1: $Y=0$.

If $\lambda=0$, Equation~\eqref{E:abcd} with $Y=0$ says that
$ad-bc=0$, and hence the vectors $aB+bZ$ and $cB+dZ$ are linearly
dependent, so we can write $\g{h}=\R(T+X)\oplus\R(cB+dZ)$. Now,
from~\eqref{E:XY} we get $cX=0$. If $c=0$ then
$\g{h}=\R(T+X)\oplus\g{g}_{2\alpha}$ and we are in
case~(\ref{L:dim2:a}), whereas if $X=0$ we are in
case~(\ref{L:dim2:b}).

If $\lambda\neq 0$, we get $c=0$ from~\eqref{E:lc} and therefore we can write
$\g{h}=\R(T+aB+X)\oplus\g{g}_{2\alpha}$.
It is obvious that in this case the orbit $H \cdot o$
is $2$-dimensional. Hence, if $\Sigma$ is a section of the action with $o \in \Sigma$, we must have
$T_o\Sigma=\{v\in\g{p}:\langle
v,\xi\rangle=\langle v,\eta\rangle=0\}$.
For $X=0$ we have $T_o\Sigma=\g{p}_\alpha$, and Corollary~\ref{C:criterion} implies
$0=\langle Z,[(1-\theta)U,(1-\theta)JU]\rangle
=\langle Z,[U,JU]\rangle
=\lVert U\rVert^2$
for all $U \in \g{g}_\alpha$, which is impossible.
Therefore we must have $X\neq 0$, and then
$T_o\Sigma=\R((1-\theta)JX)\oplus\R(-\lVert
X\rVert^2B+a(1-\theta)X)$. Since $\Sigma$ is totally geodesic, $T_o\Sigma$
is either real or complex, and this can happen only if $a = 0$, which
implies case~(\ref{L:dim2:a}).

Case 2: $Y\neq 0$.

As $X$ and $Y$ are orthogonal and $[T,Y]$ is orthogonal to $Y$,
Equation~\eqref{E:XY} implies $\lambda = \frac{a}{2}$ and
$[T,Y]=\frac{c}{2}X$. Since the connected subgroup $K_0 \cong U(1)$ of
$S(U(1)U(2))$ with Lie algebra $\g{k}_0$ acts transitively on the
unit circle in $\g{g}_\alpha$, it follows that $[T,Y]\neq 0$ and
hence also $c\neq 0$ and $X\neq 0$. From~\eqref{E:lc} we get $\lambda=0$ (which implies that $\g{h}$
is abelian) and thus also $a = 0$,
and from~\eqref{E:abcd}
we then get $\langle [X,Y],Z\rangle=2bc$. Since $X,Y\neq 0$ and
$\dim\g{g}_\alpha=2$ we also get $b\neq 0$. Finally, since $b,c\neq 0$ we can renormalize $T$ and $Y$ so that $b=1$ and $c = 2$, thus getting~(\ref{L:dim2:c}).
\end{proof}

The next step is to show that the actions arising from
Lemma~\ref{L:dim2} are orbit equivalent to the actions described in
items~(\ref{It:cohom2:a})
or~(\ref{It:cohom2:g2a}) of the Main Theorem. We have three different
possibilities:

\subsubsection*{\textup{(\ref{L:dim2:a})} $\g{h}=\R(T+X)\oplus\g{g}_{2\alpha}$
with $0 \neq T \in \g{k}_0$ and $0 \neq X\in\g{g}_\alpha$} \hfill

Since $T \neq 0$ and $\ad(T)$ is skewsymmetric, we have $[T,[T,X]] = - \rho X$ for some $\rho > 0$.
We define $g=\Exp(-\frac{1}{\rho}[T,X]) \in G$. Then we get $\Ad(g)Z = Z$ and, since $[[T,X],T+X] = \rho X + [[T,X],X]$,
\[
\Ad(g)(T+X) = T + X - X - \frac{1}{\rho}[[T,X],X] + \frac{1}{2\rho}[[T,X],X] = T - \frac{1}{2\rho}[[T,X],X].
\]
Since $[[T,X],X] \in \g{g}_{2\alpha}$ this implies  $\Ad(g)(\g{h})=\g{k}_0\oplus\g{g}_{2\alpha}$. It follows that the action is conjugate to the one in~(\ref{It:cohom2:g2a}) of the Main Theorem.

\subsubsection*{\textup{(\ref{L:dim2:b})} $\g{h}=\g{k}_0\oplus\R(aB+bZ)$
with $a,b\in\R$, $a \neq 0$ or $b \neq 0$}\hfill

If $a=0$ we get $\g{h}=\g{k}_0\oplus\g{g}_{2\alpha}$, which is case~(\ref{It:cohom2:g2a}) of the Main
Theorem. Thus we can assume $a\neq 0$. In this case we define
$g=\Exp(\frac{b}{a}Z)$. Since $[\g{k}_0,\g{g}_{2\alpha}]=0$ we get
$\Ad(g)\g{k}_0=\g{k}_0$, and since $[B,Z] = Z$ we get $\Ad(g)(aB+bZ)=aB$. Altogether this implies
$\Ad(g)\g{h}=\g{k}_0\oplus\g{a} = \g{g}_0$, and therefore the action is conjugate to the one in~(\ref{It:cohom2:a})
of the Main Theorem.

\subsubsection*{\textup{(\ref{L:dim2:c})} $\g{h} = \R(T + [T,Y] +Z) \oplus \R(2B+Y+dZ)$ with $d \in \R$, $0 \neq T \in \g{k}_0$ and $0 \neq Y\in\g{g}_\alpha$ such that $[[T,Y],Y] = 2Z$}\hfill

We define $g=\Exp(Y+\frac{d}{2}Z)$. Then
\begin{align*}
\Ad(g)(T+[T,Y]+Z)
&=T+[T,Y]+Z+[Y,T] + [Y,[T,Y]] + \frac{1}{2}[Y,[Y,T]] = T,\\
\Ad(g)(B+Y+dZ)
&=2B+Y+dZ + 2[Y,B] + d[Z,B] = 2B,
\end{align*}
and therefore $\Ad(g)\g{h}=\g{k}_0\oplus\g{a} = \g{g}_0$. Consequently the action is conjugate to the one in~(\ref{It:cohom2:a})
of the Main Theorem.

\medskip
Altogether we have proved

\begin{proposition}\label{P:cohom2:1dimsingular}
Assume that $H$ acts polarly and without fixed points on $\C H^2$ with cohomogeneity $2$ and with a $1$-dimensional singular orbit. Then the Lie algebra of $H$ is conjugate to $\g{g}_0$ or $\g{k}_0\oplus\g{g}_{2\alpha}$.
\end{proposition}

In order to finish the proof of the Main Theorem if remains to show
that the actions of the groups whose Lie algebras are
$\g{g}_0$ or $\g{k}_0\oplus\g{g}_{2\alpha}$ are indeed polar. We
use the criterion given in Corollary~\ref{C:criterion}.

\subsubsection*{Case 1: $H$ is the connected Lie subgroup of $SU(1,2)$ whose
Lie algebra is $\g{h}=\g{g}_0$}\hfill

We consider the submanifold $\Sigma=\exp_o(\g{s})$ with
$\g{s}=(1-\theta)(\g{g}_\alpha^{\R} \oplus\g{g}_{2\alpha})$. Here, $\exp_o$
denotes the  exponential map $T_o\C H^2 \to \C H^2$, and we are identifying $T_o\C
H^2$ with $\g{p}$ as usual. It is clear that $\g{s}$ is a real subspace of $\g{p}$, and
hence $\Sigma$ is a totally geodesic real hyperbolic plane $\R H^2 \subset \C H^2$.

Obviously,
$T_o\Sigma=\g{s}\subset\g{p}_\alpha\oplus\g{p}_{2\alpha}=\nu_o(H\cdot
o)$. If $K_0 \cong U(1)$ denotes the connected Lie group of $SU(1,2)$ whose Lie
algebra is $\g{k}_0 \cong \g{u}(1)$, the slice representation of $H$ at $o$ is
the representation of $K_0$ on $\g{p}_\alpha\oplus\g{p}_{2\alpha}$,
which is equivalent to the sum of the standard representation of
$U(1)$ on $\g{p}_\alpha \cong \C$, and the trivial representation on $\g{p}_{2\alpha} \cong \R$. Thus $\g{s}$ is a
section of the slice representation. Since
$[\g{s},\g{s}]=(1+\theta)[\theta\g{g}_\alpha^{\R},\g{g}_{2\alpha}] \subset \g{g}_{-\alpha} \oplus \g{g}_\alpha$, which obviously is perpendicular to $\g{g}_0 = \g{h}$, it now follows from Corollary~\ref{C:criterion} that
the action of $H$ on $\C H^2$ is polar.

\subsubsection*{Case 2: $H$ is the connected Lie subgroup of $SU(1,2)$ whose
Lie algebra is $\g{h}=\g{k}_0\oplus\g{g}_{2\alpha}$}\hfill

In this case we consider $\Sigma=\exp_o(\g{s})$ with
$\g{s}=\g{a}\oplus (1-\theta)(\g{g}_\alpha^{\R})$. Again, $\g{s}$ is a real subspace of $\g{p}$ and $\Sigma$ is a totally geodesic
$\R H^2 \subset \C H^2$.
Moreover, we have
$T_o\Sigma=\g{s}\subset\g{a}\oplus\g{p}_\alpha=\nu_o(H\cdot o)$, and
the slice representation of $H$ at $o$ is the representation of
$K_0$ on $\g{a}\oplus\g{p}_\alpha$, which is equivalent to the sum of the standard representation of
$U(1)$ on $\g{p}_\alpha \cong \C$, and the trivial representation on $\g{a} \cong \R$. Therefore $\g{s}$ is a
section of the slice representation. Finally,
$[\g{s},\g{s}]=(1+\theta)\g{g}_\alpha^{\R} \subset \g{g}_{-\alpha} \oplus \g{g}_\alpha$ is orthogonal to $\g{h} = \g{k}_0 \oplus \g{g}_{2\alpha}$, and
thus it follows from Corollary~\ref{C:criterion} that the action of
$H$ on $\C H^2$ is polar.

\medskip
Altogether we have proved the Main Theorem.

%%%%%%%%%%%%%%%%%%%%%%%%%% Bibliography %%%%%%%%%%%%%%%%%%%%%%%%%%%%%%%

\end{document}